\documentclass[1p,fleqn,preprint]{elsarticle}

\usepackage{amsmath,amssymb,amsfonts,amsthm}
\usepackage[mathscr]{eucal}
\usepackage{ifpdf}
\ifpdf
\usepackage[bookmarks=false,%
pdfstartview=FitH,linkbordercolor={0.5 1 1},%
citebordercolor={0.5 1 0.5},unicode,%
hyperindex]{hyperref}%
\else
\usepackage{breakurl}                          
\fi

\biboptions{numbers,sort&compress}

\newtheorem{theorem}{Theorem}
\newtheorem{lemma}{Lemma}
\newtheorem{example}{Example}
\newtheorem{corollary}{Corollary}
\newtheorem{remark}{Remark}
\newtheorem*{daitheorem}{Dai's Theorem}

\newcommand{\setA}{\mathscr{A}}
\newcommand{\setM}{\mathscr{M}}
\newcommand{\setP}{\mathscr{P}}

\let\emptyset\varnothing

\journal{Linear Algebra and Its Applications}

\begin{document}

\begin{frontmatter}

\title{Matrix products with constraints on the sliding block relative frequencies
of different factors\tnoteref{rfbr}}

\tnotetext[rfbr]{Supported by the Russian Foundation for Basic Research,
Project No. 13-01-13105.}

\author{Victor Kozyakin}

\address{Institute for Information Transmission
Problems\\ Russian Academy of Sciences\\ Bolshoj Karetny lane 19, Moscow
127994 GSP-4, Russia}

\ead{kozyakin@iitp.ru}
\ead[url]{http://www.iitp.ru/en/users/46.htm}

\begin{abstract}
One of fundamental results of the theory of joint/generalized spectral
radius, the Berger--Wang theorem, establishes equality between the joint
and generalized spectral radii of a set of matrices. Generalization of this
theorem on products of matrices whose factors are applied not arbitrarily
but are subjected to some constraints is connected with essential
difficulties since known proofs of the Berger--Wang theorem rely on the
arbitrariness of appearance of different matrices in the related matrix
products. Recently, X.~Dai~\cite{Dai:JFI14} proved an analog of the
Berger--Wang theorem for the case when factors in matrix products are
formed by some Markov law.

We extend the concepts of joint and generalized spectral radii to products
of matrices with constraints on the sliding block relative frequencies of
occurrences of different factors, and prove an analog of the Berger--Wang
theorem for this case.
\end{abstract}


\end{frontmatter}

\section{Introduction}\label{S-intro}

In various theoretic and applied problems there arise the matrix products
\begin{equation}\label{E-switchsys}
M_{\alpha_{0}}M_{\alpha_{1}}\cdots M_{\alpha_{n}},\quad n\ge0,
\end{equation}
where $M_{i}$ are ($d\times d$)-matrices from a finite collection
$\setM=\{M_{1},M_{2},\ldots,M_{r}\}$ with the elements from the field
$\mathbb{K}=\mathbb{R},\mathbb{C}$ of real or complex numbers, and
$\boldsymbol{\alpha}=(\alpha_{n})$ is a sequence of symbols from the set
$\setA=\{1,2,\ldots,r\}$, see,
e.g.,~\cite{Dai:JFI14,Prot:CDC05-2,SWMWK:SIAMREV07,JPB:LAA08,GugZen:LNM14,Koz:IITP13}
and the bibliography therein.

The question about the rate of growth of the matrix products
\eqref{E-switchsys} is relatively simple (at least theoretically) in edge
cases, for example, when the sequence $\boldsymbol{\alpha}=(\alpha_{n})$ is
periodic or in \eqref{E-switchsys} all possible sequences
$\boldsymbol{\alpha}=(\alpha_{n})$ with symbols from $\setA=\{1,2,\ldots,r\}$
are considered. In the latter case the question about the rate of growth of
all possible matrix products \eqref{E-switchsys} can be answered in terms of
the so-called joint or generalized spectral radii of the set of matrices
$\setM$ \cite{BerWang:LAA92,DaubLag:LAA92,DaubLag:LAA01,RotaStr:IM60}.

In intermediate situations, when the sequences
$\boldsymbol{\alpha}=(\alpha_{n})$ in \eqref{E-switchsys} are relatively
complex but still not totally arbitrary, the question about the rate of
growth of the matrix products \eqref{E-switchsys} becomes highly nontrivial
and its resolution essentially depends on the structure of the index
sequences $\boldsymbol{\alpha}=(\alpha_{n})$. In particular, one of the key
features of the index sequence in \eqref{E-switchsys} is the frequency
$p_{i}$ of occurrences of the index $i$.

As a rule, the frequency $p_{i}$ is defined as the limit of the relative
frequencies $p_{i,n}$ of occurrences of the symbol $i$ among the first $n$
members of a sequence. However, one should bear in mind that such a
definition of frequency is rather subtle and not very constructive from the
point of view of mathematical formalism. Already in the situation when one
deals with a single sequence $\boldsymbol{\alpha}=(\alpha_{n})$, this
definition is not enough informative since it does not answer the question of
how often different symbols appear in intermediate, not tending to infinity,
finite segments of a sequence. This definition becomes all the less
satisfactory in situations when one should deal with not a single sequence
but with an infinite collection of such sequences. The principal deficiency
here is that the definition of frequency given above does not withstand
transition to the limit with respect to different sequences which results in
substantial theoretical and conceptual difficulties.

To give ``good'' properties (from the point of view of ability to use
mathematical methods) to determination of frequency one often needs either to
require some kind of uniformity of convergence of the relative frequencies
$p_{i,n}$ to $p_{i}$ or to treat appearance of the related symbols in a
sequence as a realization of events generated by some random or deterministic
ergodic system, and so on. In the latter case, for the analysis of behavior
of the matrix products \eqref{E-switchsys}, the so-called multiplicative
ergodic theorem (in a probabilistic or a measure theoretic setting) is  most
often used, see, e.g., \cite{DHX:SIAMJCO08,DHX:AUT11}. However, under such an
approach one needs to impose rather strong restrictions on the laws of
forming the index sequences $\boldsymbol{\alpha}=(\alpha_{n})$ which are
often difficult to verify of confirm in applications. In all such cases the
arising families of the index sequences and of the related matrix products
can be rather attractive from the purely mathematical point of view but their
description becomes less and less constructive. In applications, it often
leads to emergence of an essential conceptual gap or of some kind strained
interpretation at use of the related objects and constructions.

At the present, there exists an extensive literature devoted to studying the
frequency properties of various classes of symbolical sequences, see, e.g.,
\cite{MacKay03,MOS:GTS99,MOS:IEEETIT01,Kitchens98,Immink04,LindMarcus95} and
the references therein. The problem on constructive determination of classes
of symbolic sequences with prescribed frequency properties and of the related
matrix products has got lesser attention. Only recently, in
\cite{Dai:JFI14,Koz:LAA14} the area of applicability of the methods of
joint/generalized spectral radius was expanded on the matrix products
\eqref{E-switchsys} in which the index sequences
$\boldsymbol{\alpha}=(\alpha_{n})$ are described by the topological Markov
chains. In connection with this, the aim of the article is to describe a
class of $r$-symbolic sequences for which it is possible to constructively
define ``approximate relative frequencies'' of symbols in the \emph{sliding
$\ell$-blocks}, that is, in each block of $\ell$ consecutive symbols. Such a
description will be given in terms of the $\ell$-step topological Markov
chains (subshifts of finite type) which allows to use the approach developed
in \cite{Dai:JFI14,Koz:LAA14} to define the joint/generalized spectral radius
for the products of matrices whose factors are applied not arbitrarily but
are subjected to some frequency constraints.

Outline the structure of the work. In Section~\ref{S-Symbseq}, we recall
basic definitions related to symbolic sequences. Then we introduce the
concept of symbolic sequences with constraints on the sliding $\ell$-block
relative frequencies, and study their principal properties. In
Theorem~\ref{T-Markov} we show that the left shift on the set of symbolic
sequences with constraints on the sliding $\ell$-block relative frequencies
is an $\ell$-step topological Markov chain. In spite of its obviousness,
Theorem~\ref{T-Markov} is of principal importance in further constructions.
Properties of symbolic sequences with constraints on the sliding block
relative frequencies obtained in Section~\ref{S-Symbseq} are used in
Section~\ref{S-MatProd} to analyze the rate of growth of norms of matrix
products. For this purpose, we recall the basics of the theory of
joint/generalized spectral radius and the results from
\cite{Dai:JFI14,Koz:LAA14} related to the generalization of the Berger--Wang
theorem on the case of Markovian products of matrices. In the final part of
Section~\ref{S-MatProd}, we define the concepts of joint and generalized
spectral radii for matrix products with constraints on the sliding
$\ell$-block relative frequencies, and deduce with the help of
Theorem~\ref{T-Markov} from the results of~\cite{Dai:JFI14,Koz:LAA14} an
analog of the Berger--Wang theorem for this case, Theorem~\ref{T-BWgen}.

\section{Symbolic sequences}\label{S-Symbseq}

Let us recall basic definitions related to symbolic sequences, see, e.g.,
\cite{LindMarcus95}. Consider infinite sequences
$\boldsymbol{\alpha}=(\alpha_{n})$ defined for
$n\in\mathbb{N}:=\{0,1,2,\ldots\}$\footnote{In symbolic dynamics, see, e.g.,
\cite{KatokHas:e,Kitchens98,LindMarcus95}, sequences defined for all integer
values of $i$, that is, for $i\in\mathbb{Z}$, are also often considered. In
this work sequences of such a type will not be needed.} and taking values
from the alphabet (set of symbols) $\setA=\{1,2,\ldots,r\}$, and denote by
$\setA^{\mathbb{N}}$ the totality of all such sequences. By $\setA_{l}$ we
denote the set of all finite sequences $\boldsymbol{\alpha}=(\alpha_{n})$ of
length $l$, and by $\setA^{*}=\cup_{l\ge1}\setA_{l}$ we denote the set of all
finite sequences (of all possible lengths) taking values in $\setA$. The
number of elements of a sequence $\boldsymbol{\alpha}\in\setA^{*}$, called
the \emph{length} of $\boldsymbol{\alpha}$, is denoted by
$|\boldsymbol{\alpha}|$, and $|\boldsymbol{\alpha}|_{i}$ stands for the
number of occurrences of the symbol $i$ in $\boldsymbol{\alpha}$. The
quantity $\frac{|\boldsymbol{\alpha}|_{i}}{|\boldsymbol{\alpha}|}$ is then
the \emph{relative frequency} of occurrences of the symbol $i$ in
$\boldsymbol{\alpha}$.

\subsection{Sequences with constraints on the sliding block relative
frequencies of symbols}\label{S-ConstrainedSeq}

Let $p=(p_{i},p_{2},\ldots,p_{r})$ be a set of positive numbers satisfying
$\sum_{i=1}^{r}p_{i}=1$, and let
\begin{equation}\label{E-fbouns}
p^{-}=(p^{-}_{1},p^{-}_{2},\ldots,p^{-}_{r}),\quad
p^{+}=(p^{+}_{1},p^{+}_{2},\ldots,p^{+}_{r}),
\end{equation}
be sets of lower and upper bounds for $p$:
\begin{equation}\label{E-freq-cond}
0\le p^{-}_{i}<p_{i}<p^{+}_{i}\le 1,\quad i=1,2,\ldots,r.
\end{equation}

Given a natural number $\ell$, denote by $\setA_{\ell}(p^{\pm})$ the set of
all finite sequences $\boldsymbol{\alpha}\in\setA_{\ell}$ for which the
relative frequencies of occurrences of different symbols satisfy
\begin{equation}\label{E-quotas}
p^{-}_{i}\le\frac{|\boldsymbol{\alpha}|_{i}}{|\boldsymbol{\alpha}|}
\le p^{+}_{i},\quad i=1,2,\ldots,r.
\end{equation}
By rewriting these last inequalities in the form
\[
p^{-}_{i}|\boldsymbol{\alpha}|\le|\boldsymbol{\alpha}|_{i}
\le p^{+}_{i}|\boldsymbol{\alpha}|,\quad i=1,2,\ldots,r,
\]
we can interpret them as constraints on the number of occurrences of
different symbols in the sequences $\boldsymbol{\alpha}\in\setA_{\ell}$. By
$\setA^{\mathbb{N}}_{\ell}(p^{\pm})$ we will denote the set of all sequences
from $\setA^{\mathbb{N}}$ each finite subsequence $\boldsymbol{\alpha}$ of
which of length $\ell$ satisfies the constraints \eqref{E-quotas}.

\begin{example}\label{Ex1}\rm Let $r=3, \ell=10$ and
$p= (0.23, 0.33, 0.44)$. Define the sets \eqref{E-fbouns} of lower and upper
bounds for $p$ by setting $p^{-}_{i}=p_{i}-0.1$ and $p^{+}_{i}=p_{i}+0.1$ for
$i=1,2,3$, that is,
\[
p^{-}=(0.13, 0.23, 0.34),\quad
p^{+}=(0.33, 0.43, 0.54).
\]
Then the set $\setA^{\mathbb{N}}_{\ell}(p^{\pm})$ contains the following
sequences:
\begin{align*}
\boldsymbol{\alpha}_{1}&=(2,1,2,3,3,3,2,3,3,1,2,1,2,3,1,3,2,3,3,3,2,1,2,2,1,3,3,1,3,3,2,\ldots),\\
\boldsymbol{\alpha}_{2}&=(3,2,2,1,3,3,3,3,2,1,2,1,2,3,3,2,3,3,2,1,2,1,3,1,3,2,3,3,2,2,1,\ldots),\\
\boldsymbol{\alpha}_{3}&=(1,1,3,3,2,2,1,2,3,3,1,3,1,3,2,2,3,2,2,3,1,3,1,3,2,1,3,2,2,3,1,\ldots).
\end{align*}
\end{example}

\begin{remark}\label{Rem1}\rm
The set $\setA_{\ell}(p^{\pm})$ may be empty even in the case when
inequalities \eqref{E-freq-cond} are satisfied. However if
$\setA_{\ell}(p^{\pm})\neq\emptyset$ then
$\setA^{\mathbb{N}}_{\ell}(p^{\pm})\neq\emptyset$, too. To show this it
suffices to take an arbitrary sequence
$\boldsymbol{\alpha}=(\alpha_{0},\alpha_{1},\ldots,\alpha_{\ell-1})\in
\setA_{\ell}(p^{\pm})$ and to observe that its periodic extension to the
right (with period $\ell$) belongs to $\setA^{\mathbb{N}}_{\ell}(p^{\pm})$.
\end{remark}

\begin{remark}\label{Rem2}\rm
In general, the frequencies of occurrences of the symbols $i=1,2,\ldots,r$ in
the sequences from $\setA^{\mathbb{N}}_{\ell}(p^{\pm})$ are not well defined.
More precisely it means the following. Denote by
$\boldsymbol{\alpha}_{n}=(\alpha_{0},\alpha_{1},\ldots,\alpha_{n})$ the
initial interval of length $n$ for an arbitrary sequence
$\boldsymbol{\alpha}=(\alpha_{0},\alpha_{1},\ldots)\in
\setA^{\mathbb{N}}_{\ell}(p^{\pm})$. For each $i=1,2,\ldots,r$, the quantity
$\frac{|\boldsymbol{\alpha}_{n}|_{i}}{|\boldsymbol{\alpha}_{n}|}$ is the
relative frequency of occurrences of the symbol $i$ among the first $n$
members of $\boldsymbol{\alpha}$. Then the relative frequencies
$\frac{|\boldsymbol{\alpha}_{n}|_{i}}{|\boldsymbol{\alpha}_{n}|}$ are
``close'' to the corresponding quantities $p_{i}$ but, in general, they may
have no limits as $n\to\infty$.
\end{remark}

The next lemma answers the question when the set $\setA_{\ell}(p^{\pm})$ is
nonempty. Recall that, for a real number $x$, the floor value $\lfloor
x\rfloor$ is the largest integer not greater than $x$ and the ceiling value
$\lceil x \rceil$ is the smallest integer not less than $x$.

\begin{lemma}\label{L-SigmaM-nonempty}
$\setA_{\ell}(p^{\pm})\neq\emptyset$ if and only if
\begin{gather}\label{E-SigmaM-nonempty1}
\lceil p^{-}_{i}\ell\rceil \le \lfloor p^{+}_{i}\ell\rfloor, \quad i=1,2,\ldots,r,\\
\label{E-SigmaM-nonempty2} \sum_{i=1}^{r}\lceil p^{-}_{i}\ell\rceil \le \ell
\le \sum_{i=1}^{r}\lfloor p^{+}_{i}\ell\rfloor.
\end{gather}
\end{lemma}

\begin{proof}
Condition \eqref{E-SigmaM-nonempty1} means that, for each $i=1,2,\ldots,r$,
there must exist at least one sequence $\boldsymbol{\alpha}$ of length $\ell$
the number $|\boldsymbol{\alpha}|_{i}$ of occurrences of the symbol $i$ in
which satisfies the ``frequency constraints'' \eqref{E-quotas}. Whereas
condition \eqref{E-SigmaM-nonempty2}, provided that condition
\eqref{E-SigmaM-nonempty1} is satisfied, is equivalent to the existence of at
least one sequence $\boldsymbol{\alpha}$ of length $\ell$ the number
$|\boldsymbol{\alpha}|_{i}$ of occurrences of each symbol $i=1,2,\ldots,r$ in
which satisfies \eqref{E-quotas}.
\end{proof}

\begin{example}\label{Ex2}\rm
Let $r=3, \ell=10$, and the sets \eqref{E-fbouns} of lower and upper bounds
for $p$ be the same as in Example~\ref{Ex2}. Then
\[
\lceil p^{-}_{1}\ell\rceil =2,\quad \lceil p^{-}_{2}\ell\rceil =3,\quad\lceil p^{-}_{3}\ell\rceil =4,\quad
\lfloor p^{+}_{1}\ell\rfloor =3,\quad \lfloor p^{+}_{2}\ell\rfloor =4,\quad\lfloor p^{+}_{3}\ell\rfloor =5,
\]
and $\sum_{i=1}^{r}\lceil p^{-}_{i}\ell\rceil =9 \le \ell \le
12=\sum_{i=1}^{r}\lfloor p^{+}_{i}\ell\rfloor$. So, both conditions
\eqref{E-SigmaM-nonempty1} and \eqref{E-SigmaM-nonempty2} hold.

Let again $p= (0.23, 0.33, 0.44)$, but the sets of lower and upper bounds for
$p$ this time be defined by the equalities $p^{-}_{i}=p_{i}-0.01$ and
$p^{+}_{i}=p_{i}+0.01$ for $i=1,2,3$, that is,
\[
p^{-}=(0.22, 0.32, 0.43),\quad
p^{+}=(0.24, 0.34, 0.45).
\]
In this case
\[
\lceil p^{-}_{1}\ell\rceil =3,\quad \lceil p^{-}_{2}\ell\rceil =4,\quad\lceil p^{-}_{3}\ell\rceil =5,\quad
\lfloor p^{+}_{1}\ell\rfloor =2,\quad \lfloor p^{+}_{2}\ell\rfloor =3,\quad\lfloor p^{+}_{3}\ell\rfloor =4,
\]
and conditions \eqref{E-SigmaM-nonempty1} are not valid for any $i=1,2,3$.

At last, let again $p= (0.23, 0.33, 0.44)$, but the sets of lower and upper
bounds for $p$ this time be defined as $p^{-}_{i}=p_{i}-0.05$ and
$p^{+}_{i}=p_{i}+0.05$ for $i=1,2,3$, that is,
\[
p^{-}=(0.18, 0.28, 0.39),\quad
p^{+}=(0.28, 0.38, 0.49).
\]
Then
\[
\lceil p^{-}_{1}\ell\rceil =2,\quad \lceil p^{-}_{2}\ell\rceil =3,\quad\lceil p^{-}_{3}\ell\rceil =4,\quad
\lfloor p^{+}_{1}\ell\rfloor =2,\quad \lfloor p^{+}_{2}\ell\rfloor =3,\quad\lfloor p^{+}_{3}\ell\rfloor =4,
\]
and conditions \eqref{E-SigmaM-nonempty1} hold for each $i=1,2,3$ whereas
condition \eqref{E-SigmaM-nonempty2} is not valid because
$\sum_{i=1}^{r}\lfloor p^{+}_{i}\ell\rfloor =9<\ell$.
\end{example}

From Lemma~\ref{L-SigmaM-nonempty} and Example~\ref{Ex2} it is seen that
$\setA_{\ell}(p^{\pm})\neq\emptyset$ if and only if the ``gap'' between the
related quantities $p^{-}_{i}$ and $p^{+}_{i}$ is not ``too small''.

\begin{lemma}\label{L-condperiod}
Let the sets of quantities $p^{-}$ and $p^{+}$ satisfy
\eqref{E-SigmaM-nonempty1}, and also let one of the following conditions be
valid:
\begin{itemize}
\item[(i)] in \eqref{E-SigmaM-nonempty2} one of the inequalities turns to a
    strict equality, that is,
\[
\sum_{i=1}^{r}\lceil p^{-}_{i}\ell\rceil =
\ell\quad\text{or}\quad\sum_{i=1}^{r}\lfloor p^{+}_{i}\ell\rfloor =\ell;
\]
\item[(ii)] in \eqref{E-SigmaM-nonempty2} both of inequalities are strict,
    that is,
\begin{equation}\label{E-SigmaM-nonempty3}
\sum_{i=1}^{r}\lceil p^{-}_{i}\ell\rceil < \ell < \sum_{i=1}^{r}\lfloor p^{+}_{i}\ell\rfloor,
\end{equation}
while in \eqref{E-SigmaM-nonempty1} all inequalities, with the exception
possibly of one, turn to equalities.
\end{itemize}

Then all the sequences from $\setA^{\mathbb{N}}_{\ell}(p^{\pm})$ are
periodic.
\end{lemma}

\begin{proof} Let $\boldsymbol{\alpha}$ be a sequence from
$\setA_{\ell}(p^{\pm})$ (such a sequence exists by
Lemma~\ref{L-SigmaM-nonempty}). Then
\begin{equation}\label{E-sumalphai}
|\boldsymbol{\alpha}|=\sum_{i=1}^{r}|\boldsymbol{\alpha}|_{i}=\ell,
\end{equation}
and inequalities \eqref{E-quotas} imply
\begin{equation}\label{Eineqalphai}
\lceil p^{-}_{i}\ell\rceil \le|\boldsymbol{\alpha}|_{i} \le \lfloor
p^{+}_{i}\ell\rfloor, \quad i=1,2,\ldots,r.
\end{equation}

Now observe that, if in condition~(i) the first equality holds, then
\eqref{E-sumalphai} and \eqref{Eineqalphai} imply
\begin{equation}\label{E-eqforalpha1}
|\boldsymbol{\alpha}|_{i} = \lceil p^{-}_{i}\ell\rceil, \quad i=1,2,\ldots,r,
\end{equation}
whereas, if in condition~(i) the second equality holds, then
\eqref{E-sumalphai} and \eqref{Eineqalphai} imply
\begin{equation}\label{E-eqforalpha2}
|\boldsymbol{\alpha}|_{i} = \lfloor p^{+}_{i}\ell\rfloor, \quad i=1,2,\ldots,r.
\end{equation}

At last, let condition~(ii) be valid. Then in \eqref{E-SigmaM-nonempty1}
$r-1$ inequalities of $r$ are, in fact, equalities. Without loss of
generality one can think that in this case the first $r-1$ of inequalities
\eqref{E-SigmaM-nonempty1} are equalities, and then from \eqref{E-sumalphai}
and \eqref{Eineqalphai} it follows that
\begin{align}\label{E-eqforalpha31}
|\boldsymbol{\alpha}|_{i} &= \lceil p^{-}_{i}\ell\rceil=\lfloor p^{+}_{i}\ell\rfloor, \quad i=1,2,\ldots,r-1,\\
\label{E-eqforalpha32}
|\boldsymbol{\alpha}|_{r} &= \sum_{i=1}^{r-1}|\boldsymbol{\alpha}|_{i}.
\end{align}

So, under any of conditions~(i) or (ii), the quantities
$|\boldsymbol{\alpha}|_{i}$ are defined uniquely by one of relations
\eqref{E-eqforalpha1}--\eqref{E-eqforalpha32} and do not depend on the choice
of $\boldsymbol{\alpha}\in\setA_{\ell}(p^{\pm})$. In this case in each
sequence $\boldsymbol{\alpha}'\in\setA^{\mathbb{N}}_{\ell}(p^{\pm})$, for any
$\ell+1$ successive symbols
$(\alpha'_{k},\ldots,\alpha'_{k+\ell-1},\alpha'_{k+\ell})$, the equality
$\alpha'_{k}=\alpha'_{k+\ell}$ takes place, that is, the sequence
$\boldsymbol{\alpha}'$ is $\ell$-periodic.
\end{proof}

By Lemma~\ref{L-condperiod} aperiodic behavior of sequences from
$\setA^{\mathbb{N}}_{\ell}(p^{\pm})$ may happen only in the case when the
``gap'' between the related quantities $p^{-}_{i}$ and $p^{+}_{i}$ is large
enough to conditions \eqref{E-SigmaM-nonempty1} be valid and in conditions
\eqref{E-SigmaM-nonempty2} both inequalities be strict.

\begin{lemma}\label{L-aperiod-seq}
Let the sets of quantities $p^{-}$ and $p^{+}$ satisfy the system of
inequalities \eqref{E-SigmaM-nonempty1} at least two of which are strict, and
also let condition \eqref{E-SigmaM-nonempty3} hold. Then for each finite
sequence
$\boldsymbol{\alpha}=(\alpha_{0},\alpha_{1},\ldots,\alpha_{\ell-1})\in
\setA_{\ell}(p^{\pm})$ there exist at least two different infinite sequences
$\boldsymbol{\alpha}'=(\alpha'_{n})_{n\ge0}\in
\setA^{\mathbb{N}}_{\ell}(p^{\pm})$ and
$\boldsymbol{\alpha}''=(\alpha''_{n})_{n\ge0}\in
\setA^{\mathbb{N}}_{\ell}(p^{\pm})$ whose initial subsequences of length
$\ell$ coincide with $\boldsymbol{\alpha}$, that is,
$\alpha'_{n}=\alpha''_{n}=\alpha_{n}$ for $n=0,1,\ldots,\ell-1$.
\end{lemma}

\begin{proof}
Denote by $I:=\{i_{1},\ldots, i_{m}\}$ the set of all $i\in\{1,2,\ldots,r\}$
for which the inequalities $\lceil p^{-}_{i}\ell\rceil \le \lfloor
p^{+}_{i}\ell\rfloor$ in \eqref{E-SigmaM-nonempty1} are strict, that is,
\begin{align*}
\lceil p^{-}_{i}\ell\rceil &< \lfloor p^{+}_{i}\ell\rfloor,\quad i=i_{1},\ldots, i_{m},\\
\lceil p^{-}_{i}\ell\rceil &= \lfloor p^{+}_{i}\ell\rfloor,\quad i\neq i_{1},\ldots, i_{m}.
\end{align*}
Then by the lemma conditions $m\ge2$.

Let $\boldsymbol{\alpha}=(\alpha_{0},\alpha_{1},\ldots,\alpha_{\ell-1})$ be a
sequence from $\setA_{\ell}(p^{\pm})$; such a sequence exists by
Lemma~\ref{L-SigmaM-nonempty}. Observe that there exist $i',i''\in I$ such
that
\begin{equation}\label{E-i-ii}
\lceil p^{-}_{i'}\ell\rceil < |\boldsymbol{\alpha}|_{i'}\le\lfloor p^{+}_{i'}\ell\rfloor,\quad
\lceil p^{-}_{i''}\ell\rceil \le |\boldsymbol{\alpha}|_{i''}<\lfloor p^{+}_{i''}\ell\rfloor.
\end{equation}
Indeed, if for all $i\in I$ the equality $|\boldsymbol{\alpha}|_{i}=\lceil
p^{-}_{i}\ell\rceil$ was valid, then the equality $\sum_{i=1}^{r}\lceil
p^{-}_{i}\ell\rceil = |\boldsymbol{\alpha}|=\ell$ would also be valid, which
contradicts to \eqref{E-SigmaM-nonempty3}. And if for all $i\in I$ the
equality $|\boldsymbol{\alpha}|_{i}=\lfloor p^{+}_{i}\ell\rfloor$ was valid,
then the equality  $\sum_{i=1}^{r}\lfloor p^{+}_{i}\ell\rfloor =
|\boldsymbol{\alpha}|=\ell$ would be valid, which also contradicts to
\eqref{E-SigmaM-nonempty3}.

Now observe that without loss of generality we can suppose that the first
symbol in
$\boldsymbol{\alpha}=(\alpha_{0},\alpha_{1},\ldots,\alpha_{\ell-1})$
coincides with $i'$, that is, $\alpha_{0}=i'$. Indeed, in the opposite case
we can extend the sequence $\boldsymbol{\alpha}$ to the right by periodicity
with period $\ell$, and instead of $\boldsymbol{\alpha}$ take such an
interval $(\alpha_{k},\alpha_{1},\ldots,\alpha_{k+\ell-1})$ of the extended
sequence for which the equality $\alpha_{k}=i'$ holds .

So, let $\alpha_{0}=i'$. Then construct two sequences
\[
(\alpha'_{0},\alpha'_{1},\ldots,\alpha'_{\ell-1},\alpha'_{\ell}),\quad
(\alpha''_{0},\alpha''_{1},\ldots,\alpha''_{\ell-1},\alpha''_{\ell})
\]
whose initial intervals of length $\ell$ coincide with $\boldsymbol{\alpha}$
while the members with the indices $\ell$ are as follows: $\alpha'_{\ell}=i'$
and $\alpha''_{\ell}=i''$. In virtue of \eqref{E-i-ii}, both variants of
extension of the sequence $\boldsymbol{\alpha}$ lead to that the sequences
$(\alpha'_{1},\ldots,\alpha'_{\ell-1},\alpha'_{\ell})$ and
$(\alpha''_{1},\ldots,\alpha''_{\ell-1},\alpha''_{\ell})$ are distinct and
belong to $\setA_{\ell}(p^{\pm})$. By extending them to the right to infinite
sequences from $\setA^{\mathbb{N}}_{\ell}(p^{\pm})$ (for example, by
$\ell$-periodicity), we obtain the required sequences $\boldsymbol{\alpha}',
\boldsymbol{\alpha}''$.
\end{proof}

\begin{corollary}\label{Cor1}
Under the conditions of Lemma~\ref{L-aperiod-seq} the set
$\setA^{\mathbb{N}}_{\ell}(p^{\pm})$ contains infinitely many different
aperiodic sequences.
\end{corollary}

Denote by $\setA^{*}_{\ell}(p^{\pm})$ the set of all finite sequences each of
which is an initial interval of a sequence from
$\setA^{\mathbb{N}}_{\ell}(p^{\pm})$. Clearly, the set
$\setA^{*}_{\ell}(p^{\pm})$ consists either of sequences of length lesser
than $\ell$ allowing extension to the right to sequences from
$\setA_{\ell}(p^{\pm})$ or of sequences of length greater than or equal to
$\ell$ each finite subsequence of which of length $\ell$ belongs to
$\setA_{\ell}(p^{\pm})$.

\begin{lemma}\label{L-inv+subadd}
The following assertions are valid:
\begin{itemize}
  \item[(i)] If $\boldsymbol{\alpha}=(\alpha_{n})_{n=0}^{\infty}\in
      \setA^{\mathbb{N}}_{\ell}(p^{\pm})$, then
      $\boldsymbol{\alpha}_{m}=(\alpha_{n+m})_{n=0}^{\infty}\in
      \setA^{\mathbb{N}}_{\ell}(p^{\pm})$ for any $m\ge0$.
  \item[(ii)] If $\boldsymbol{\alpha}=(\alpha_{0},\ldots,\alpha_{k})\in
      \setA^{*}_{\ell}(p^{\pm})$, then for any $m=1,2,\ldots,k$ the
      inclusions $\boldsymbol{\alpha}'=(\alpha_{0},\ldots,\alpha_{m-1})\in
      \setA^{*}_{\ell}(p^{\pm})$ and
      $\boldsymbol{\alpha}''=(\alpha_{m},\ldots,\alpha_{k})\in
      \setA^{*}_{\ell}(p^{\pm})$ hold.
\end{itemize}
\end{lemma}

The assertions of Lemma~\ref{L-inv+subadd} are obvious and so their proofs
are omitted. Assertion~(i) is equivalent to the \emph{invariance of the set
$\setA^{\mathbb{N}}_{\ell}(p^{\pm})$ with respect to the left shifts}, that
is, to the fact that together with each sequence
$\boldsymbol{\alpha}=(\alpha_{n})_{n\in\mathbb{N}}$ the set
$\setA^{\mathbb{N}}_{\ell}(p^{\pm})$ contains also the sequence
$\boldsymbol{\alpha}'=(\alpha'_{n})_{n\in\mathbb{N}}$ defined by the
equalities $\alpha'_{n}=\alpha_{n+1}$ for $n\in\mathbb{N}$. Assertion~(i)
implies that, in definition of the set $\setA^{*}_{\ell}(p^{\pm})$, one could
consider all finite intervals (not only the initial ones) of the sequences
from $\setA^{\mathbb{N}}_{\ell}(p^{\pm})$. The property expressed by
assertion~(ii) will be referred to as the \emph{sub-additivity of the set
$\setA^{*}_{\ell}(p^{\pm})$}.

\begin{remark}\label{R-biblio}\rm
The author failed to find in the literature any explicit references to the
symbolic sequences with the allowed $\ell$-blocks of type
$\setA_{\ell}(p^{\pm})$. Nevertheless, symbolic sequences with similar
properties arise in various applications and theoretical studies. The
symbolic sequences with the allowed $\ell$-blocks of type
$\setA_{\ell}(p^{\pm})$  may be treated as the ``constrained sequences''
arising in the theory of the so-called constrained noiseless channels
\cite[Ch.~17]{MacKay03}. The frequency properties of such sequences with
specific allowed forbidden patterns of symbols was studied, e.g., in
\cite{ChoiSpz:ISIT01}, and in \cite{MOS:GTS99,MOS:IEEETIT01} the theory of
joint spectral radius was used for these goals. Conceptually close are also
the so-called $(d,k)$-constrained sequences arising in the method of RLL
coding (runlength-limited coding) used for data processing in mass data
storage systems such as hard drives, compact discs and so on, see, e.g.,
\cite{Kitchens98,Immink04}.

The symbolic sequences with the allowed $\ell$-blocks of type
$\setA_{\ell}(p^{\pm})$ are similar to the so-called $k$-balanced sequences
\cite{Kitchens98,LindMarcus95}, although the former constitute a broader
class. The symbolic sequences with the allowed $\ell$-blocks of type
$\setA_{\ell}(p^{\pm})$ may have no limiting frequencies for some symbols
which differentiate them from the symbolic sequences whose relative
frequencies are uniformly convergent \cite{FerMont:CANT10}.
\end{remark}

\subsection{Subshifts of finite type}\label{SS-shifts}

In this section we show that the set of infinite sequences
$\setA^{\mathbb{N}}_{\ell}(p^{\pm})$ can be naturally treated as the
so-called \emph{subshifts of finite type} (or \emph{$\ell$-step topological
Markov chains}). Recall the necessary definitions following
to~\cite{KatokHas:e,LindMarcus95}.

As usual the operator $\sigma: \boldsymbol{\alpha}\mapsto
\boldsymbol{\alpha}'$ transferring a sequence
$\boldsymbol{\alpha}=(\alpha_{n})_{n\in\mathbb{N}}\in\setA^{\mathbb{N}}$ to
the sequence
$\boldsymbol{\alpha}'=(\alpha'_{n})_{n\in\mathbb{N}}\in\setA^{\mathbb{N}}$
defined by the equalities $\alpha'_{n}=\alpha_{n+1}$ for $n\in\mathbb{N}$ is
called the \emph{left shift} or simply \emph{shift} on $\setA^{\mathbb{N}}$.
Given a square matrix $\omega=(\omega_{ij})_{i,j=1}^{r}$ of order $r$ with
the elements from the set $\{0,1\}$, we define
\[
\setA^{\mathbb{N}}_{\omega}:=\{\boldsymbol{\alpha}=(\alpha_{n})\in \setA^{\mathbb{N}}:
\omega_{\alpha_{n}\alpha_{n+1}}=1,~n\in\mathbb{N}\}.
\]
In other words, the matrix $\omega$ determines all transitions between the
symbols of the alphabet $\setA=\{1,2,\ldots,r\}$ in the sequences from
$\setA^{\mathbb{N}}_{\omega}$. The restriction of the shift operator $\sigma$
to $\setA^{\mathbb{N}}_{\omega}$ is called the \emph{topological Markov
chain} defined by the matrix $\omega$ of allowed
transitions~\cite{KatokHas:e,LindMarcus95}. In this case $\sigma$ is also
often referred to as a \emph{subshift of finite type}\footnote{Sometimes the
terms topological Markov chain or subshift of finite type are applied not to
the shift operator $\sigma$ but to the set of the sequences
$\setA^{\mathbb{N}}_{\varOmega}$.}.

There is a natural class of symbolic systems more general than the
topological Markov chains. Given a map $\varOmega: \setA^{\ell+1}\to
\{0,1\}$, define
\[
\setA^{\mathbb{N}}_{\varOmega}:=\{\boldsymbol{\alpha}=(\alpha_{n})\in \setA^{\mathbb{N}}:
\varOmega(\alpha_{n},\ldots,\alpha_{n+\ell})=1,~n\in\mathbb{N}\}.
\]
The restriction of the shift map $\sigma$ to $\setA^{\mathbb{N}}_{\varOmega}$
is called the \emph{$\ell$-step topological Markov chain} defined by the
function $\varOmega$ of allowed transitions.

From the point of view of dynamics, the $\ell$-step topological Markov chains
are the same as the usual ($1$-step) topological Markov chains since the
former can be described as the topological Markov chains with the alphabet
$\setA=\{1,2,\ldots,r\}^{\ell}$ and the transition matrix $\omega$ such that
\begin{equation}\label{E-transmat}
\omega_{(\alpha_{1},\ldots,\alpha_{\ell}),(\alpha'_{1},\ldots,\alpha'_{\ell})}=1,
\end{equation}
if $\alpha'_{k}=\alpha_{k+1}$ for $k=1,\ldots,\ell-1$ and
$\varOmega(\alpha_{1},\ldots,\alpha_{\ell},\alpha'_{\ell})=1$, see, e.g.,
\cite[Sect.~1.9]{KatokHas:e}.

\begin{theorem}\label{T-Markov}
Let, for some sets of quantities $p^{-}=(p^{-}_{1},\ldots,p^{-}_{r})$ and
$p^{+}=(p^{+}_{1},\ldots,p^{+}_{r})$ satisfying \eqref{E-freq-cond}, the
conditions of Lemma~\ref{L-SigmaM-nonempty} hold. The the restriction of the
shift $\sigma$ to the set $\setA^{\mathbb{N}}_{\ell}(p^{\pm})$ is an
$\ell$-step topological Markov chain.
\end{theorem}

\begin{proof}
It suffices to note that
$\setA^{\mathbb{N}}_{\ell}(p^{\pm})=\setA^{\mathbb{N}}_{\varOmega}$, where
the map $\varOmega: \setA^{\ell+1}\to \{0,1\}$ is such that
$\varOmega(\alpha_{n},\ldots,\alpha_{n+\ell})=1$ if and only if
$(\alpha_{n},\ldots,\alpha_{n+\ell-1}),
(\alpha_{n+1},\ldots,\alpha_{n+\ell})\in\setA_{\ell}(p^{\pm})$, and
$\varOmega(\alpha_{n},\ldots,\alpha_{n+\ell})=0$ in the opposite case.
\end{proof}

Despite its evidence, Theorem~\ref{T-Markov} is of principal importance since
it allows to treat the symbolic sequences with constraints on the sliding
block relative frequencies of symbols as the ($\ell$-step) topological Markov
chains.

\section{Matrix products with constraints on the sliding block
relative frequencies of factors}\label{S-MatProd}

In this section we strengthen the results of \cite{Dai:JFI14,Koz:LAA14} on
the joint/generalized spectral radius for the Markovian matrix products for
the matrix products \eqref{E-switchsys} whose index sequences
$\boldsymbol{\alpha}=(\alpha_{n})$ are subjected to constraints on the
sliding block relative frequencies of symbols.

\subsection{The joint and generalized spectral radii}\label{S-JGSR}
Recall basic concepts of the theory of joint/generalized spectral radius.
Given a sub-multi\-plicative norm\footnote{A norm $\|\cdot\|$ on a space of
linear operators is called sub-multiplicative if $\|AB\|\le\|A\|\cdot\|B\|$
for any operators $A$ and $B$.} $\|\cdot\|$ on $\mathbb{K}^{d\times d}$, the
limit
\begin{equation}\label{E-JSRad}
\rho({\setM}):=
\limsup_{n\to\infty}\rho_{n}({\setM})\qquad \left(~ =
\lim_{n\to\infty}\rho_{n}({\setM}) =
\inf_{n\ge1}\rho_{n}({\setM})\right),
\end{equation}
where
\[
\rho_{n}({\setM}):=
\sup\left\{\|M_{\alpha_{n}}\cdots M_{\alpha_{1}}\|^{1/n}:~\alpha_{j}\in \setA\right\},
\]
is called the \emph{joint spectral radius} of the set of matrices $\setM$
\cite{RotaStr:IM60}. This limit always exists, finite and does not depend on
the norm $\|\cdot\|$. If $\setM$ consists of a single matrix, then
(\ref{E-JSRad}) turns into the known Gelfand formula for the spectral radius
of a linear operator. By this reason sometimes (\ref{E-JSRad}) is called the
generalized Gelfand formula \cite{ShihWP:LAA97}.

The \emph{generalized spectral radius} of the set of matrices $\setM$ is the
quantity defined by a similar to \eqref{E-JSRad} formula in which instead of
the norm it is taken the spectral radius $\rho(\cdot)$ of the corresponding
matrices \cite{DaubLag:LAA92,DaubLag:LAA01}:
\begin{equation}\label{E-GSRad}
\hat{\rho}({\setM}):=
\limsup_{n\to\infty}\hat{\rho}_{n}({\setM})\qquad \left(~ = \sup_{n\ge1}\hat{\rho}_{n}({\setM})\right),
\end{equation}
where
\[
\hat{\rho}_{n}({\setM}):=\sup\left\{\rho(M_{\alpha_{n}}\cdots M_{\alpha_{1}})^{1/n}:~\alpha_{j}\in \setA\right\}.
\]

As has been proved by M.~Berger and Y.~Wang~\cite{BerWang:LAA92} the
quantities $\rho({\setM})$ and $\hat{\rho}({\setM})$ coincide with each other
for bounded sets of matrices $\setM$:
\begin{equation}\label{E-bergWang}
\hat{\rho}({\setM})=\rho({\setM}).
\end{equation}
This formula has numerous applications in the theory of joint/generalized
spectral radius. It implies the continuous dependence of the
joint/ge\-neral\-ized spectral radius on $\setM$. From \eqref{E-bergWang} it
also follows that the quantities $\hat{\rho}_{n}({\setM})$ and
$\rho_{n}({\setM})$, for any $n$, form the lower and upper bounds
respectively for the joint/generalized spectral radius of $\setM$:
\begin{equation}\label{E-GSR-JSR}
\hat{\rho}_{n}({\setM})\le
\hat{\rho}({\setM})=\rho({\setM})\le
\rho_{n}({\setM}),
\end{equation}
which is often used to evaluate the accuracy of computation of the
joint/ge\-neral\-ized spectral radius.

\subsection{The Markovian joint and generalized spectral radii}\label{S-MarkovJSR}
The distinctive feature of the definitions \eqref{E-JSRad} and
\eqref{E-GSRad} is that the products of matrices $M_{\alpha_{n}}\cdots
M_{\alpha_{1}}$ in them correspond to all possible sequences of indices
$\boldsymbol{\alpha}=(\alpha_{1},\ldots,\alpha_{n})$. A more complicated
situation is when these matrix products are somehow constrained, for example,
some combinations of matrices in them are forbidden. One of situations of the
kind was investigated in \cite{Dai:JFI14,Koz:LAA14}, where the concepts of
the Markovian joint and generalized spectral radii were introduced to analyze
the matrix products with constraints of the Markovian type on the neighboring
matrices. Another situation of the kind is described in what follows.

Let $\omega=(\omega_{ij})$ be an $(r\times r)$-matrix with the elements from
the set $\{0,1\}$. A finite sequence $(\alpha_{1},\ldots,\alpha_{n})$ with
the elements from $\setA$ will be called \emph{$\omega$-admissible} if
$\omega_{\alpha_{j+1}\alpha_{j}}=1$ for all $1\le j \le n-1$ and besides
there exists $\alpha_{*}\in\setA$ such that
$\omega_{\alpha_{*}\alpha_{n}}=1$. Denote by $W_{r,\omega}$ the set of all
$\omega$-admissible finite sequences. The matrix products
$M_{\alpha_{n}}\cdots M_{\alpha_{1}}$ corresponding to the
$\omega$-admissible sequences $(\alpha_{1},\ldots,\alpha_{n})$ will be called
\emph{Markovian} since such products of matrices arise naturally in the
theory of matrix cocycles over the topological Markov chains, see, e.g.,
\cite{KatokHas:e,Kitchens98}.

Define analogs of formulae \eqref{E-JSRad} and \eqref{E-GSRad} for the
$\omega$-admissible products of matrices. The limit
\begin{equation}\label{E-JSRadM}
\rho({\setM},\omega):=
\limsup_{n\to\infty}\rho_{n}({\setM},\omega),
\end{equation}
where
\[
\rho_{n}({\setM},\omega):=\sup\left\{\|M_{\alpha_{n}}\cdots M_{\alpha_{1}}\|^{1/n}:~
(\alpha_{1},\ldots,\alpha_{n})\in W_{r,\omega}\right\},
\]
is called the \emph{Markovian joint spectral radius} of the set of matrices
$\setM$ defined by the \emph{matrix of admissible transitions $\omega$}. If,
for some $n$, the set of all $\omega$-admissible sequences
$(\alpha_{1},\ldots,\alpha_{n})$ is empty then we put
$\rho_{n}({\setM},\omega)=0$. In this case, the sets of all
$\omega$-admissible sequences $(\alpha_{1},\ldots,\alpha_{k})$ will be also
empty for each $k\ge n$, and therefore $\rho({\setM},\omega)=0$. The question
on existence of arbitrarily long $\omega$-admissible sequences can be
answered algorithmically in a finite number of steps. In particular, the set
$W_{r,\omega}$ has arbitrarily long sequences if each column of the matrix
$\omega$ contains at least one nonzero element.

Likewise to formula \eqref{E-JSRad}, the limit \eqref{E-JSRadM} always
exists, finite and does not depend on the norm $\|\cdot\|$. Moreover, by the
Fekete Lemma~\cite{Fekete:MZ23} (see also \cite[Ch.~3, Sect.~1]{PolyaSezgoI})
the sub-multiplicativity in $n$ of the quantity
$\rho_{n}^{n}({\setM},\omega)$ implies the existence of
$\lim_{n\to\infty}\rho_{n}({\setA},\omega)$ and of
$\inf_{n\ge1}\rho_{n}({\setA},\omega)$ and their equality to the limit
\eqref{E-JSRadM}:
\[
\rho({\setM},\omega):=
\limsup_{n\to\infty}\rho_{n}({\setM},\omega)=\lim_{n\to\infty}\rho_{n}({\setM},\omega)=
\inf_{n\ge1}\rho_{n}({\setM},\omega).
\]

The quantity
\begin{equation}\label{E-GSRadM} \hat{\rho}({\setM},\omega):=
\limsup_{n\to\infty}\hat{\rho}_{n}({\setM},\omega),
\end{equation}
where
\[
\hat{\rho}_{n}({\setM},\omega):=\sup\left\{\rho(M_{\alpha_{n}}\cdots M_{\alpha_{1}})^{1/n}:~
(\alpha_{1},\ldots,\alpha_{n})\in W_{r,\omega}\right\},
\]
is called the \emph{Markovian generalized spectral radius} of the set of
matrices $\setM$ defined by the matrix of admissible transitions $\omega$.
Here again we are putting $\hat{\rho}_{n}({\setM},\omega)=0$ if the set of
$\omega$-admissible sequences of indices $(\alpha_{1},\ldots,\alpha_{n})$ is
empty. Like in the case of formula \eqref{E-GSRad}, the limit
\eqref{E-GSRadM} coincides with $\sup_{n\ge1}\hat{\rho}_{n}({\setM},\omega)$.

For the Markovian products of matrices there are valid inequalities
\begin{equation}\label{E-MGSR-MJSR}
\hat{\rho}_{n}({\setM},\omega)\le
\hat{\rho}({\setM},\omega)\le\rho({\setM},\omega)\le
\rho_{n}({\setM},\omega),
\end{equation}
similar to \eqref{E-GSR-JSR}. However the question whether there is a valid
equality
\begin{equation}\label{E-bergWangMark}
\hat{\rho}({\setM},\omega)=\rho({\setM},\omega),
\end{equation}
similar to the Berger--Wang equality~\eqref{E-bergWang}, becomes more
complicated. To answer it we will need one more definitions.

An $\omega$-admissible finite sequence $(\alpha_{1},\ldots,\alpha_{n})$ will
be called \emph{periodically extendable} if
$\omega_{\alpha_{1}\alpha_{n}}=1$. Not every $\omega$-admissible finite
sequence can be periodically extended. However, if arbitrarily long
$\omega$-admissible sequences exist, then periodically extendable
$\omega$-admissible sequences exist too. The set of all periodically
extendable $\omega$-admissible sequences is denoted by
$W^{(\text{per})}_{r,\omega}$.

Define the quantity
\[
\hat{\rho}^{(\text{per})}_{n}({\setM},\omega):=\sup\left\{\rho(M_{\alpha_{n}}\cdots M_{\alpha_{1}})^{1/n}:~
(\alpha_{1},\ldots,\alpha_{n})\in W^{(\text{per})}_{r,\omega}\right\},
\]
and set\footnote{Like in the definitions of the Markovian joint and
generalized spectral radii we put
$\hat{\rho}^{(\text{per})}_{n}({\setM},\omega)=0$ if the set of all the
periodically extendable sequences of length $n$ is empty.}
\[
\hat{\rho}^{(\text{per})}({\setM},\omega):=
\limsup_{n\to\infty}\hat{\rho}^{(\text{per})}_{n}({\setM},\omega).
\]

\begin{daitheorem}(See~\cite{Dai:JFI14})
$\hat{\rho}^{(\text{per})}({\setM},\omega)=\rho({\setM},\omega)$.
\end{daitheorem}

Since $W^{(\text{per})}_{r,\omega}\subseteq W_{r,\omega}$ then
$\hat{\rho}^{(\text{per})}_{n}({\setM},\omega)\le
\hat{\rho}_{n}({\setM},\omega)$ for each $n\ge 1$, and therefore
$\hat{\rho}^{(\text{per})}({\setM},\omega)\le \hat{\rho}({\setM},\omega)$.
This last inequality together with \eqref{E-MGSR-MJSR} by Dai's Theorem then
implies the Markovian analog~\eqref{E-bergWangMark} of the Berger--Wang
formula~\eqref{E-bergWang}.

\subsection{The constrained joint and generalized spectral radii}\label{S-ConstrainedJSR}

Let, likewise in the previous section, $\setM=\{M_{1},M_{2},\ldots,M_{r}\}$
be a finite set of matrices with the elements from the field
$\mathbb{K}=\mathbb{R},\mathbb{C}$ of real or complex numbers. Let also
$\ell$ be a natural number and $p_{i}$, $p^{\pm}_{i}$, $i=1,2,\ldots,r$, be
sets of quantities satisfying \eqref{E-freq-cond}.

A finite sequence $\boldsymbol{\alpha}=(\alpha_{1},\ldots,\alpha_{n})$ with
the elements from $\setA$ will be referred to as
\emph{$\setA_{\ell}(p^{\pm})$-admissible} if
$\boldsymbol{\alpha}\in\setA^{*}_{\ell}(p^{\pm})$, that is, if each its
subsequence $(\alpha_{j},\ldots,\alpha_{j+\ell-1})$ of length $\ell$ belongs
to $\setA_{\ell}(p^{\pm})$, and each its subsequence of length lesser than
$\ell$ allows extension to the right to a sequence from
$\setA_{\ell}(p^{\pm})$. Denote by $W_{\setA_{\ell}(p^{\pm})}$ the set of all
finite $\setA_{\ell}(p^{\pm})$-admissible sequences
$\boldsymbol{\alpha}=(\alpha_{1},\ldots,\alpha_{n})$. The problem on
existence of arbitrarily long $\setA_{\ell}(p^{\pm})$-admissible sequences
has been resolved in Lemmata~\ref{L-SigmaM-nonempty}--\ref{L-inv+subadd}.
Products of matrices $M_{\alpha_{n}}\cdots M_{\alpha_{1}}$ corresponding to
the $\setA_{\ell}(p^{\pm})$-admissible sequences
$(\alpha_{1},\ldots,\alpha_{n})$ will be referred to as
\emph{$\setA_{\ell}(p^{\pm})$-admissible}.

Now, the concepts of joint and generalized spectral radii for
$\setA_{\ell}(p^{\pm})$-admissible products of matrices from $\setM$ can be
defined by almost literal repetition of the related definitions from the
previous section. The limit
\begin{equation}\label{E-JSRadConstrained}
\rho({\setM},\setA_{\ell}(p^{\pm})):=
\limsup_{n\to\infty}\rho_{n}({\setM},\setA_{\ell}(p^{\pm})),
\end{equation}
where
\[
\rho_{n}({\setM},\setA_{\ell}(p^{\pm})):=\sup\left\{\|M_{\alpha_{n}}\cdots M_{\alpha_{1}}\|^{1/n}:~
(\alpha_{1},\ldots,\alpha_{n})\in W_{\setA_{\ell}(p^{\pm})}\right\},
\]
is called the \emph{$\setA_{\ell}(p^{\pm})$-constrained joint spectral
radius} of the set of matrices $\setM$.

Like in the cases of formulae \eqref{E-JSRad} or \eqref{E-JSRadM}, the limit
\eqref{E-JSRadConstrained} always exists, finite and does not depend on the
norm $\|\cdot\|$. Moreover, from assertion (ii) of Lemma~\ref{L-inv+subadd}
it follows that the sets $\setA^{*}_{\ell}(p^{\pm})$ possess the property of
sub-additivity and then the quantity
$\rho_{n}^{n}({\setM},\setA_{\ell}(p^{\pm}))$ is sub-multiplicative in $n$.
Therefore, by the already mentioned Fekete Lemma there exists
$\lim_{n\to\infty}\rho_{n}({\setM},\setA_{\ell}(p^{\pm}))$ coinciding with
$\rho({\setM},\setA_{\ell}(p^{\pm}))$ as well with
$\inf_{n\ge1}\rho_{n}({\setM},\setA_{\ell}(p^{\pm}))$. This means that the
$\setA_{\ell}(p^{\pm})$-constrained joint spectral radius may be defined by
each of the following equalities:
\[
\rho({\setM},\setA_{\ell}(p^{\pm})):=
\limsup_{n\to\infty}\rho_{n}({\setM},\setA_{\ell}(p^{\pm}))=\lim_{n\to\infty}\rho_{n}({\setM},\setA_{\ell}(p^{\pm}))=
\inf_{n\ge1}\rho_{n}({\setM},\setA_{\ell}(p^{\pm})).
\]

Similarly, the quantity
\[
\hat{\rho}({\setM},\setA_{\ell}(p^{\pm})):=
\limsup_{n\to\infty}\hat{\rho}_{n}({\setM},\setA_{\ell}(p^{\pm})),
\]
where
\[
\hat{\rho}_{n}({\setM},\setA_{\ell}(p^{\pm})):=\sup\left\{\rho(M_{\alpha_{n}}\cdots M_{\alpha_{1}})^{1/n}:~
(\alpha_{1},\ldots,\alpha_{n})\in W_{\setA_{\ell}(p^{\pm})}\right\},
\]
can be called the \emph{$\setA_{\ell}(p^{\pm})$-constrained generalized
spectral radius} of the set of matrices $\setM$.

For the $\setA_{\ell}(p^{\pm})$-admissible products of matrices the
inequalities
\begin{equation}\label{E-CGSR-CJSR}
\hat{\rho}_{n}({\setM},\setA_{\ell}(p^{\pm}))\le
\hat{\rho}({\setM},\setA_{\ell}(p^{\pm}))\le\rho({\setM},\setA_{\ell}(p^{\pm}))\le
\rho_{n}({\setM},\setA_{\ell}(p^{\pm})),
\end{equation}
similar to \eqref{E-GSR-JSR} or \eqref{E-MGSR-MJSR}, hold. However the
question about the validity of the equality
\begin{equation}\label{E-BergWangConstrained}
\hat{\rho}({\setM},\setA_{\ell}(p^{\pm}))=\rho({\setM},\setA_{\ell}(p^{\pm})),
\end{equation}
analogous to the Berger--Wang equality \eqref{E-bergWang}, like in the cases
of the Markovian joint and generalized spectral radii, is not so evident. Let
us call an $\setA_{\ell}(p^{\pm})$-admissible finite sequence
$(\alpha_{0},\ldots,\alpha_{n-1})$ \emph{periodically extendable} if it may
be extended to the right to an $n$-periodic sequence belonging to
$\setA^{\mathbb{N}}_{\ell}(p^{\pm})$. The set of all periodically extendable
$\setA_{\ell}(p^{\pm})$-admissible sequences will be denoted by
$W^{(\text{per})}_{\setA_{\ell}(p^{\pm})}$. Clearly, this set is nonempty
provided that $\setA_{\ell}(p^{\pm})\neq\emptyset$ because each sequence from
$\setA_{\ell}(p^{\pm})$ admits an $\ell$-periodic extension to the right.

Define, for each $n\ge 1$, the quantity
\[
\hat{\rho}^{(\text{per})}_{n}({\setM},\setA_{\ell}(p^{\pm})):=\sup\left\{\rho(M_{\alpha_{n}}\cdots M_{\alpha_{1}})^{1/n}:~
(\alpha_{1},\ldots,\alpha_{n})\in W^{(\text{per})}_{\setA_{\ell}(p^{\pm})}\right\},
\]
and set
\[
\hat{\rho}^{(\text{per})}({\setM},\setA_{\ell}(p^{\pm})):=
\limsup_{n\to\infty}\hat{\rho}^{(\text{per})}_{n}({\setM},\setA_{\ell}(p^{\pm})).
\]
Then the following generalization of the Berger--Wang formula for the case of
matrix products with constraints on the sliding block relative frequencies of
factors or, what is the same, for the case of
$\setA_{\ell}(p^{\pm})$-admissible matrix products is valid.

\begin{theorem}\label{T-BWgen}
$\hat{\rho}^{(\text{per})}({\setM},\setA_{\ell}(p^{\pm}))=\rho({\setM},\setA_{\ell}(p^{\pm}))$.
\end{theorem}

\begin{proof}
By Theorem~\ref{T-Markov} the restriction of the shift $\sigma$ on the set
$\setA^{\mathbb{N}}_{\ell}(p^{\pm})$ is an $\ell$-step topological Markov
chain. But, as was mentioned in Section~\ref{SS-shifts}, the $\ell$-step
topological Markov chains are the same as the usual ($1$-step) topological
Markov chains with the alphabet $\setA=\{1,2,\ldots,r\}^{\ell}$ and
appropriate matrices of allowed transitions \eqref{E-transmat}. Then the
claim of theorem follows from Dai's Theorem.
\end{proof}

Since $W^{(\text{per})}_{\setA_{\ell}(p^{\pm})}\subseteq
W_{\setA_{\ell}(p^{\pm})}$ then
$\hat{\rho}^{(\text{per})}_{n}({\setM},\setA_{\ell}(p^{\pm}))\le
\hat{\rho}_{n}({\setM},\setA_{\ell}(p^{\pm}))$ for each $n\ge 1$, and
therefore $\hat{\rho}^{(\text{per})}({\setM},\setA_{\ell}(p^{\pm}))\le
\hat{\rho}({\setM},\setA_{\ell}(p^{\pm}))$. From here and from
\eqref{E-CGSR-CJSR}, by Theorem~\ref{T-BWgen}, then follows the analog
\eqref{E-BergWangConstrained} of the Berger--Wang formula \eqref{E-bergWang}
for the matrix products with constraints on the sliding block relative
frequencies of factors.

\begin{remark}\rm
In this section the admissible sequences were defined as follows. It was
given the set $\setP=\setA_{\ell}(p^{\pm})$ of the sequences of length
$\ell$, and a finite sequence was treated admissible in two cases: either
when its length was less than $\ell$ and the sequence had a right-extension
to a sequence from $\setP$ or when its length was greater than or equal to
$\ell$ and each its subsequence of length $\ell$ belonged to $\setP$.

Under such a treatment of admissible sequences specific type of the set
$\setP$ is not important. Therefore all the constructions of this section
retain their validity for an arbitrary set $\setP$ and accordingly defined
admissible sequences, cf. \cite{ChoiSpz:ISIT01}. Of course, under such an
approach the question about non-emptiness of the set of all admissible
sequences should be addressed separately. The Markov sequences described in
Section~\ref{S-MarkovJSR} keep within this, wider scheme of definition of
admissible sequences, in particular.
\end{remark}

\section*{Acknowledgments}
The author is grateful to Eug\`{e}ne Asarin and Grigory Kabatiansky for the
valuable comments on the concept of symbolic sequences with frequency
constraints, and also to Xiongping Dai for fruitful discussions of a problem
of the joint and generalized spectral radii for matrix products with the
Markovian or frequency constraints on factors.

\bibliographystyle{elsarticle-num}
\bibliography{ConstrainedJSRv4}
\end{document}